\documentclass [11pt]{amsart}
\usepackage {amsmath, amssymb, amscd, mathrsfs, url ,pinlabel,hyperref, verbatim}
\usepackage[text={5.5in,9in},centering,letterpaper,dvips]{geometry}
\usepackage{color,multirow,dcpic,latexsym,pictexwd,graphicx,pdfpages}
\usepackage{fancyhdr, tikz, setspace}
\usetikzlibrary{positioning, arrows.meta}

\newtheorem {theorem}{Theorem}
\newtheorem {lemma}[theorem]{Lemma}

\theoremstyle{remark}

\newtheorem {remark}[theorem]{Remark}

\numberwithin{equation}{section}
\numberwithin{theorem}{section}

\title{More Brieskorn Spheres Bounding Rational Balls}
\author{O{\u{g}}uz \c{S}avk}
\address{Department of Mathematics, Bo\u{g}az{\i}\c{c}{\i}  University, Bebek, Istanbul, 34342, Turkey}
\email{\href{mailto:oguz.savk@boun.edu.tr}{oguz.savk@boun.edu.tr}}
\date{}

\begin{document}

\begin{abstract}
We call an integral homology sphere \emph{non-trivially} bounds a rational homology ball if it is obstructed from bounding an integral homology ball.  After Fintushel and Stern's well-known example $\Sigma(2,3,7)$, Akbulut and Larson recently provided the first infinite families of Brieskorn spheres non-trivially bounding rational homology balls: $\Sigma(2,4n+1,12n+5)$ and $\Sigma(3,3n+1,12n+5)$ for odd~$n$. Using their technique, we present new such families: $\Sigma(2,4n+3,12n+7)$ and $\Sigma(3,3n+2,12n+7)$ for even $n$. Also manipulating their main argument, we simply recover some classical results of Akbulut and Kirby, Fickle, Casson and Harer, and Stern about Brieskorn spheres bounding integral homology balls.
\end{abstract}
\maketitle

\section{Introduction}
\label{intro}

The integral homology cobordism group $\Theta_{\mathbb{Z}}^3$ has been playing a central role in both low- and high-dimensional topology \cite{M18}. Working with rational coefficients, the rational homology cobordism group $\Theta_{\mathbb{Q}}^3$ can be also taken into account. There is a canonical homomorphism $\psi: \Theta_{\mathbb{Z}}^3 \to \Theta_{\mathbb{Q}}^3$ induced by inclusion. One way to measure the complexity between homology cobordism groups goes through exploring the kernel of the map $\psi$, denoted by $\mathrm{Ker}(\psi)$. For the other perspectives of studying the map $\psi$, see \cite{KL14}, \cite{ALa18}, \cite{ACP18} and \cite{GL18}.

The natural source for integral homology spheres is Brieskorn homology spheres $\Sigma(p,q,r) = \{x^p +y^q +z^r = 0\} \cap S^5 \subset \mathbb{C}^3$ with pairwise coprime positive integers $p,q$ and $r$. Fintushel and Stern \cite{FS84} gave the first example, $\Sigma(2,3,7)$, non-trivially bounding a rational homology ball. This result can be interpreted as the non-triviality of $\mathrm{Ker}(\psi)$. As the Brieskorn sphere $\Sigma(2,3,7)$ has a non-vanishing Neumann-Siebenmann invariant $\bar{\mu}$ (\cite{N80}, \cite{S80}), this also shows the existence of an infinite order subgroup $\mathbb{Z}$ in $\mathrm{Ker}(\psi)$. The knowledge about the algebraic structure of $\mathrm{Ker}(\psi)$ is limited to those outcomes.

The Brieskorn sphere $\Sigma(2,3,7)$ has remained the single example for more than thirty years. The difficulty of finding another such examples is due to the handle decomposition of four-manifolds. If such an integral homology sphere exists, then the corresponding rational homology ball necessarily contains three-handles. 

In \cite{AL18}, Akbulut and Larson made a huge progress in this area by initially showing that the Brieskorn sphere $\Sigma(2,3,19)$ also non-trivially bounds a rational homology ball. Additionally, they provided first two infinite families of Brieskorn spheres non-trivially bounding rational homology balls: $\Sigma(2,4n+1,12n+5)$ and $\Sigma(3,3n+1,12n+5)$ for odd $n$. Using their technique, we present new such families.

\begin{theorem}
\label{main}
The Brieskorn spheres $\Sigma(2,4n+3,12n+7)$ and $\Sigma(3,3n+2,12n+7)$ bound rational homology balls.\footnote{For $n=1$, these spheres bound even contractible four-manifolds, see \cite{FS81} and \cite{F84}.} When $n$ is even, they both have $\bar{\mu}=1$ and thus they non-trivially bound rational homology balls.
\end{theorem}

It is not possible to detect linear independence of these spheres in $\Theta_{\mathbb{Z}}^3$ by calculating their current integral homology cobordism invariants.\footnote{It is also unrealizable for the infinite families of Akbulut and Larson.} If it were possible with an affirmative way, this would imply the existence of an infinitely generated subgroup $\mathbb{Z}^{\infty}$ in $\mathrm{Ker}(\psi)$.

Continuing with the manipulation of Akbulut and Larson's main argument, we simply recover some classical results proven around nineteen eighties following the celebrated work of Kirby \cite{K78}. The initial two were showed in \cite{AK79} and \cite{F84} respectively. The following two families came from \cite{CH81} and their first elements were also appeared in \cite{AK79}. The rest were due to \cite{S78}.\footnote{Originally, these spheres were all shown to be bound contractible four-manifolds.}

\begin{theorem}[Akbulut and Kirby, Fickle, Casson and Harer, and Stern]
\label{chs}
The following Brieskorn spheres bound integral homology balls: $\Sigma(2,3,13)$, $\Sigma(2,3,25)$, $\Sigma(2,4n+1,4n+3)$, $\Sigma(3,3n+1,3n+2)$, $\Sigma(2,4n+1,20n+7)$, $\Sigma(3,3n+1,21n+8)$, $\Sigma(2,4n+3,20n+13)$ and $\Sigma(3,3n+2,21n+13)$.
\end{theorem}

Both proofs are based on the combination of the following two observations. The first one is that attaching a four-dimensional two-handle $B^2 \times B^2$ to an integral (resp. a rational) homology $S^1 \times D^3$ produces an integral (resp. a rational) homology ball if the boundary of the resulting manifold is an integral homology sphere. The latter one is that a smoothly slice (resp. a rationally slice) knot bounds a smoothly embedded disk in the four-ball $B^4$ (resp. in a rational homology ball). Each integral (resp. rational) homology ball for the Brieskorn spheres is obtained by choosing an integral (resp. a rational) homology $S^1 \times D^3$ to be a smooth disk complement for a smoothly slice (resp. a rationally slice) knot and attaching a four-dimensional two-handle appropriately.

All these examples are obtained from the slice twist knots. It will be interesting to find Brieskorn spheres bounding homology balls by starting with slice knots other than twist knots. Note that the rationally slice knots in the three-sphere are in abundance due to Kawauchi \cite{Kaw79} and \cite{Kaw09}, and Kim and Wu \cite{KW18}. Indeed, any hyperbolic amphichiral knot is rationally slice. Further, any Miyazaki knot (i.e. fibered amphichiral knot with irreducible Alexander polynomial) is rationally slice.

We finally point out that all Brieskorn spheres listed in the paper are simply related to each other by the topological notion called \emph{Seifert fiber surgery}. Applying a specific type of Seifert fiber surgery along a Seifert fiber seems to effect as adding a twist to initial slice knot. However, the complete understanding of this claim is still mysterious from the handle theoretic perspective. For the details of Seifert fiber surgery operation, we suggest the recent paper of Lidman and Tweedy \cite{LT18}. 

\subsection*{Organization}
The structure of the paper is as follows. In Section \ref{rhb}, we present the proof of new Brieskorn spheres bounding rational homology balls. In Section \ref{ihb}, we give the proof of manipulation of Akbulut and Larson's argument. We finally recover some classical results about Brieskorn spheres bounding integral balls by using our new argument in a simple way.

\subsection*{Acknowledgements}
The author would like to thank his advisor \c{C}a\u gr{\i} Karakurt for his guidance, support, and tolerance. Thoughtful suggestions of the anonymous referee improve the flow and structure of the paper. The author is grateful to Jonathan Simone for recognizing a mistake in the handle diagram of surgeries of twist knots.

\section{Rational Homology Balls}
\label{rhb}
A knot in the three-sphere is called \emph{rationally slice} if it bounds a smoothly properly embedded disk in a rational homology ball. The basic example of rational slice knots is the figure-eight knot, see \cite[Theorem 4.16]{C07}, \cite[Section 3]{AL18}, and the combination of papers \cite{Kaw79} and \cite{Kaw09}. 

The main argument of Akbulut and Larson relates rationally slice knots with rational homology balls via surgeries.

\begin{lemma}[\cite{AL18}, Lemma 2]
\label{rball}
Let $Y$ be the three-manifold obtained by zero-surgery on a rationally slice knot in the three-sphere. Then any integral homology sphere obtained by an integral surgery on on a knot in $Y$ bounds a rational homology ball.
\end{lemma}

Now we are ready to prove our main theorem.

\begin{proof}[Proof of Theorem \ref{main}]
Following the recipe in \cite[Section 1.1]{Sav02}, it can be easily seen that the Brieskorn spheres $\Sigma(2,4n+3,12n+7)$ and $\Sigma(3,3n+2,12n+7)$ are respectively the boundary of the negative-definite unimodular plumbing graphs shown in Figure~\ref{fig:pl3}. To complete the proof via Lemma \ref{rball}, we use the dual approach by giving integral surgeries from their plumbing graphs.

In particular, we show that $\Sigma(2,4n+3,12n+7)$ and $\Sigma(3,3n+2,12n+7)$ are both obtained by $(-1)$-surgery on a knot in $Y$ where $Y$ is the $3$-manifold obtained by performing the zero-surgery on the figure-eight knot. Their surgery diagrams corresponding to their plumbing graphs appear in Figure~\ref{fig:2} and the dark black $(-1)$-framed components give the necessary surgery to $Y$.

\begin{figure}[ht]  \begin{center}
		\includegraphics[width=.9\textwidth]{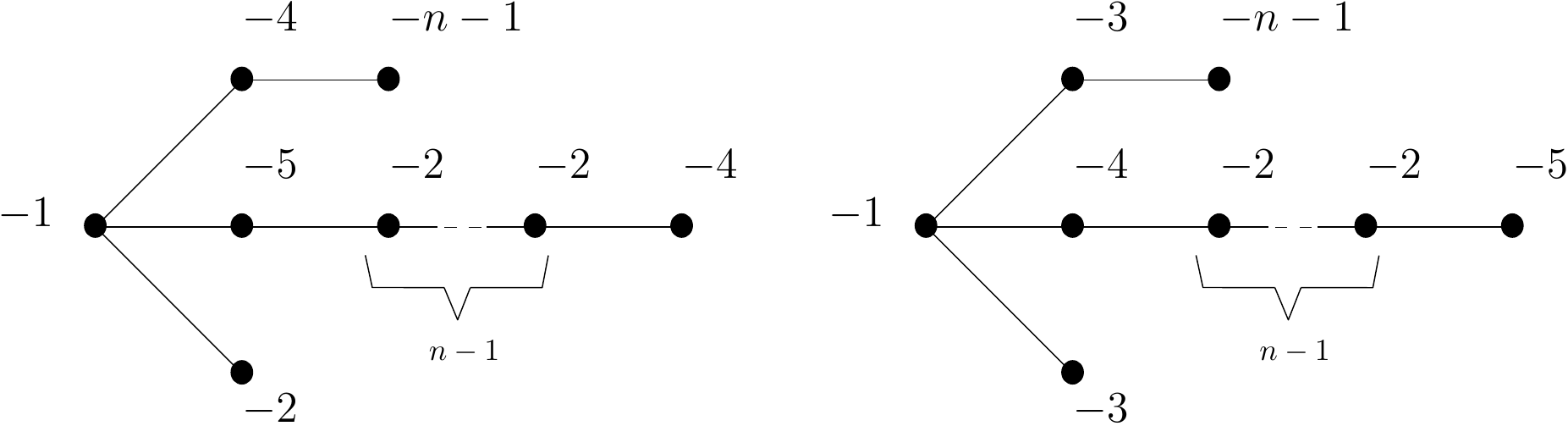}       
		\caption{The plumbing graphs}      \label{fig:pl3} 
	\end{center}
\end{figure}

We first consider $\Sigma(2,4n+3,12n+7)$ and split its proof into two parts. For the base case $n=1$, the sequence of blow downs are displayed by dark black $(-1)$-framed components and this explicit procedure can be seen in Figure~\ref{fig:1}. For the general case $n>1$, we reduce the proof to the base case by applying $n$-times blow downs and to fulfil the rest of the proof we address to the base case, see Figure~\ref{fig:2}.

\begin{figure}[ht]  \begin{center}
		\includegraphics[width=1\columnwidth]{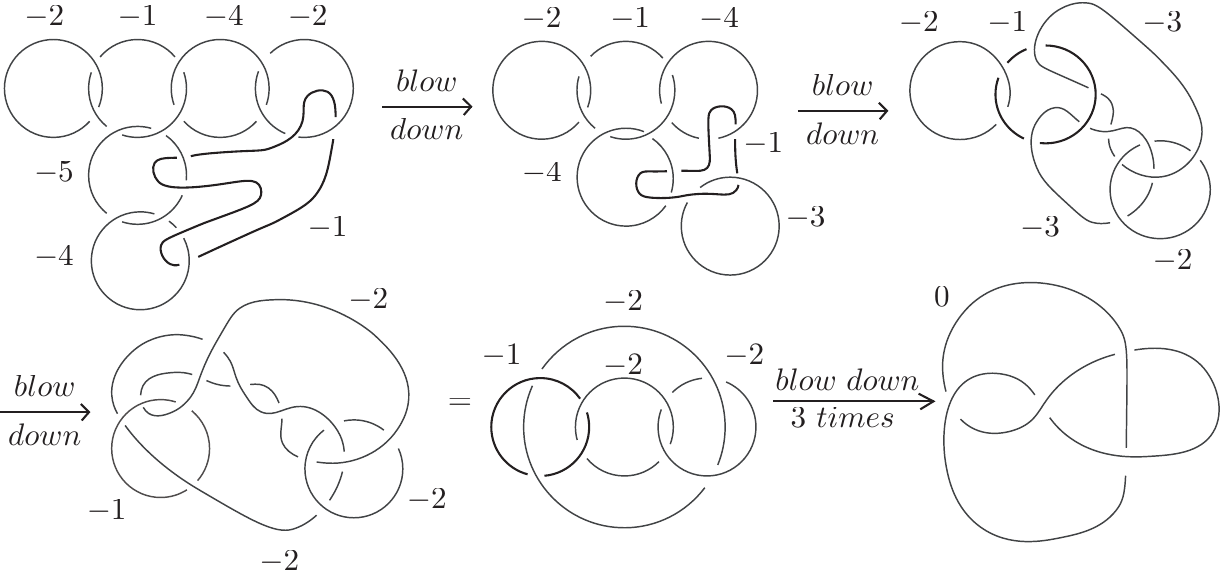}       
		\caption{The $(-1)$-surgery from $\Sigma(2,7,19)$ to $Y$}      \label{fig:1} 
	\end{center}
\end{figure}

\begin{figure}[ht]  \begin{center}
		\includegraphics[width=1\columnwidth]{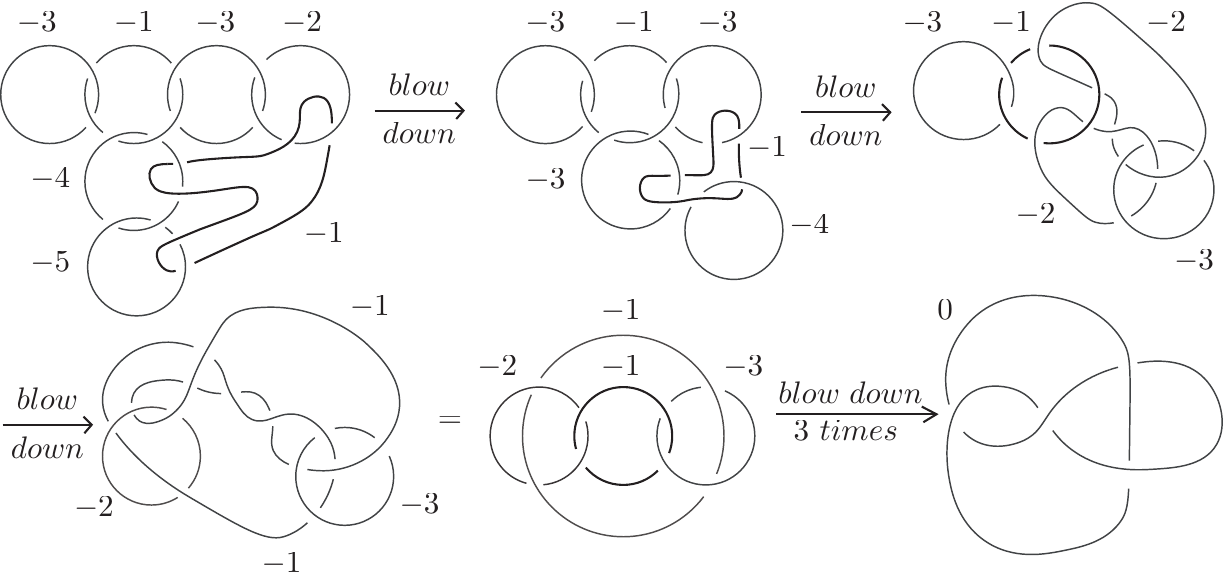}       
		\caption{The $(-1)$-surgery from $\Sigma(3,5,19)$ to $Y$}      \label{fig:1.2} 
	\end{center}
\end{figure}

\begin{figure}[ht]  \begin{center}
		\includegraphics[width=1\columnwidth]{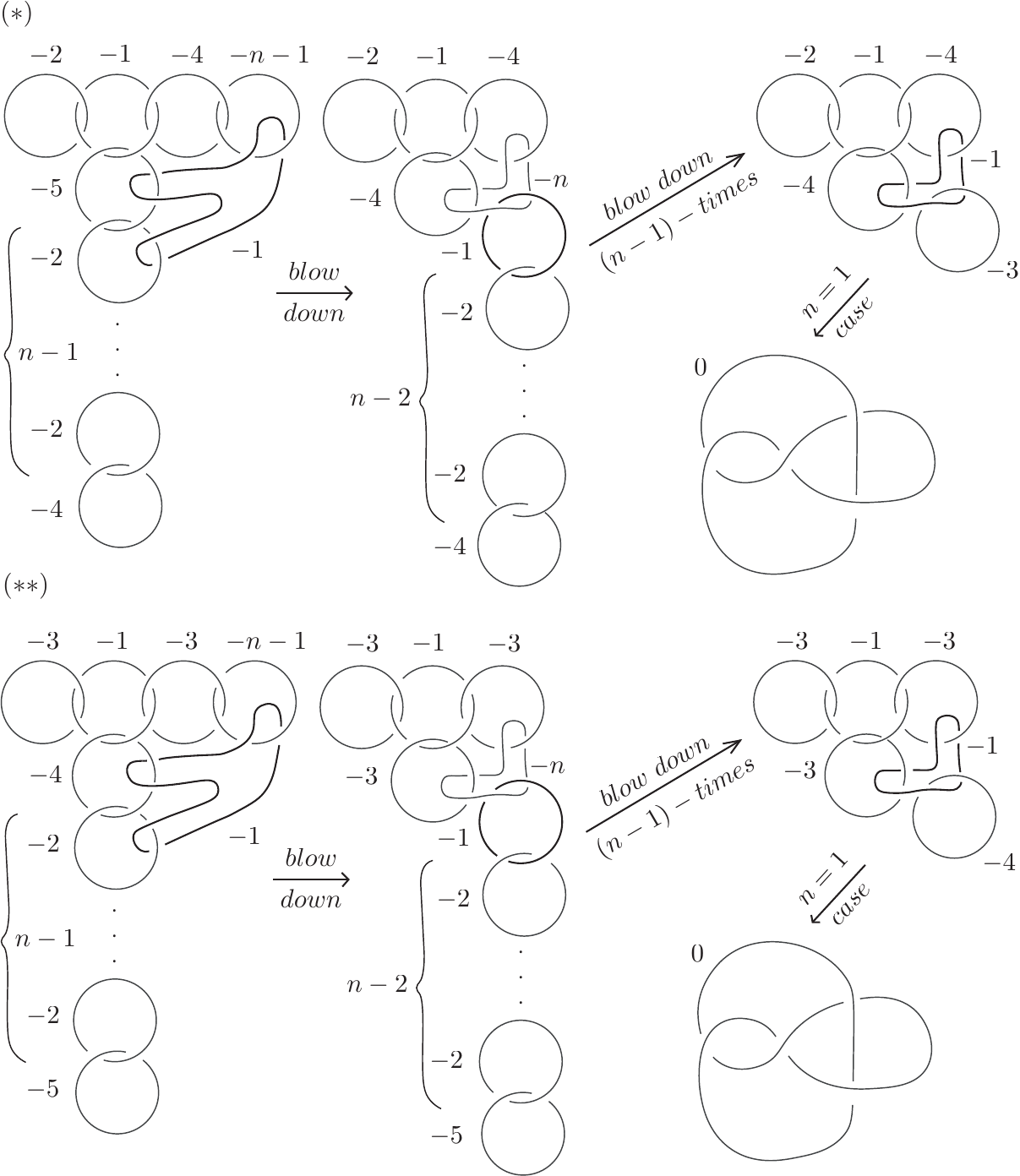}       
		\caption{The sequence of blow downs for $\Sigma(2,4n+3,12n+7)$ and $\Sigma(3,3n+2,12n+7)$ for $n >1$}      \label{fig:2} 
	\end{center}
\end{figure}

Since their plumbing graphs are quite similar, we follow the identical flow for the proof of Brieskorn spheres $\Sigma(3,3n+2,12n+7)$. For details, one can see Figure~\ref{fig:1.2} and Figure~\ref{fig:2} respectively. Since the blow down operation does not change the boundary three-manifold, we conclude that Brieskorn spheres $\Sigma(2,4n+3,12n+7)$ and $\Sigma(3,3n+2,12n+7)$ bound rational homology balls.

The computation of Neumann-Siebenmann invariant $\bar{\mu}$ (\cite{N80} and \cite{S80}) is well-known for the negative-definite unimodular plumbing graphs due to the combinatorial approach in \cite{NR78}. Note that plumbings of $\Sigma(2,4n+3,12n+7)$ and $\Sigma(3,3n+2,12n+7)$ have both signature $-n-5$,  and  when $n$ is even the square of their spherical Wu classes are both $-n-13$ (otherwise they are $-n-5$). Thus for even $n$ these Brieskorn spheres have $\bar{\mu}=1$. Therefore, they do not bound integral homology balls because the Neumann-Siebenmann invariant is an integral homology cobordism invariant for Brieskorn spheres, see \cite[Corollary 7.34]{Sav02}.
\end{proof}

\begin{remark}
We shall call the two-stage (base case $n=1$ and general case $n>1$) handle diagrammatic procedure in the proof of Theorem~\ref{main} \emph{reduction trick} and we shortly denote it by $\mathrm{RT}$.
\end{remark}

\section{Integral Homology Balls}
\label{ihb}
A knot in the three-sphere is called \emph{smoothly slice} if it bounds a smoothly properly embedded disk in the four-ball. The typical examples of smoothly slice knots are the unknot and the stevedore knot.

To associate smoothly slice knot surgeries with integral homology balls, we manipulate the main argument of Akbulut and Larson. The similar proof yields the following lemma, and for completeness we include a short proof.

\begin{lemma}
\label{integral}
Let $Y'$ be the three-manifold obtained by zero-surgery on a smoothly slice knot in the three-sphere. Then any integral homology sphere obtained by an integral surgery on on a knot in $Y'$ bounds an integral homology ball.\footnote{The lemma actually gives a contractible four-manifold if $Y'$ is $S^1 \times S^2$, which is the zero-surgery on the unknot.}
\end{lemma}

\begin{proof}
We consider the four-manifold $X= B^4 \setminus \nu D$, where $\nu D$ denotes the tubular neighborhood of the slice disk $D$. Since $\nu D$ is diffeomorphic to $D \times B^2$, we clearly have $\partial X = Y'$. Further, $X$ has the integral homology of $S^1 \times D^3$ by a simple Mayer-Vietoris argument for the triplet $(B^4, \nu D, X)$.

An integral surgery on a knot in $Y'$ corresponds to attaching a four-dimensional two-handle $B^2 \times B^2$ to $X$. Thus if the resulting three-manifold is an integral homology sphere, then the four-manifold $W$ obtained by attaching the corresponding two-handle to $X$ must be an integral homology ball. This claim can be easily seen from combining the Mayer-Vietoris sequence for the triplet $(W, X, B^2 \times B^2)$, the long exact sequence for the pair $(W, \partial W)$, and the Poincar\'e-Lefschetz duality together.
\end{proof}

We finally recover the results of Akbulut and Kirby, Fickle, Casson and Harer, and Stern in the following simple way.

\begin{proof}[Proof of Theorem \ref{chs}]
Completing the proof by Lemma \ref{integral}, we show that all these spheres are all obtained by $(-1)$-surgery on a knot in $Y'$ where $Y'$ is the zero-surgery on the unknot or the stevedore knot. 

We shall begin with single examples $\Sigma(2,3,13)$ and $\Sigma(2,3,25)$. By [Example 1.4, \cite{Sav02}], we know that $\Sigma(2,3,6n+1)$ are obtained by $(+1)$-surgery on twist knot $K_n$ for $n \geq 0$, see the first picture of Figure~\ref{fig:3}. Remark that $\Sigma(2,3,1)= S^3$, $K_0$ is the unknot and $K_2$ is the stevedore knot. Blowing down the dark black $(-1)$-framed components in the second, third and last pictures of Figure~\ref{fig:3} results in $(+1)$-surgery on $K_2$, $K_2$ and $K_4$, respectively.

\begin{figure}[ht]  \begin{center}
		\includegraphics[width=0.75\columnwidth]{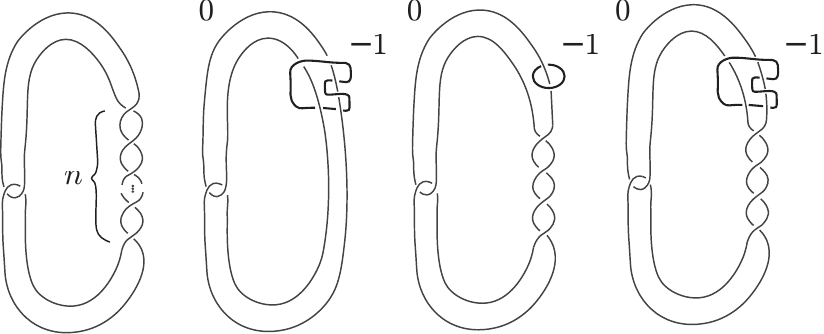}       
		\caption{The twist knot $K_n$, $\Sigma(2,3,13)$, $\Sigma(2,3,13)$, and $\Sigma(2,3,25)$ respectively}      \label{fig:3} 
	\end{center}
\end{figure}

Recap that $\Sigma(2,4n+1,4n+3)$, $\Sigma(3,3n+1,3n+2)$, $\Sigma(2,4n+1,20n+7)$, $\Sigma(3,3n+1,21n+8)$, $\Sigma(2,4n+3,20n+13)$ and $\Sigma(3,3n+2,21n+13)$ are respectively the boundary of the negative-definite unimodular plumbing graphs shown in Figure~\ref{fig:plother}. To apply Lemma~\ref{integral}, we use the dual approach by giving integral surgeries from their plumbing graphs.

\begin{figure}[ht]  \begin{center}
		\includegraphics[width=1\columnwidth]{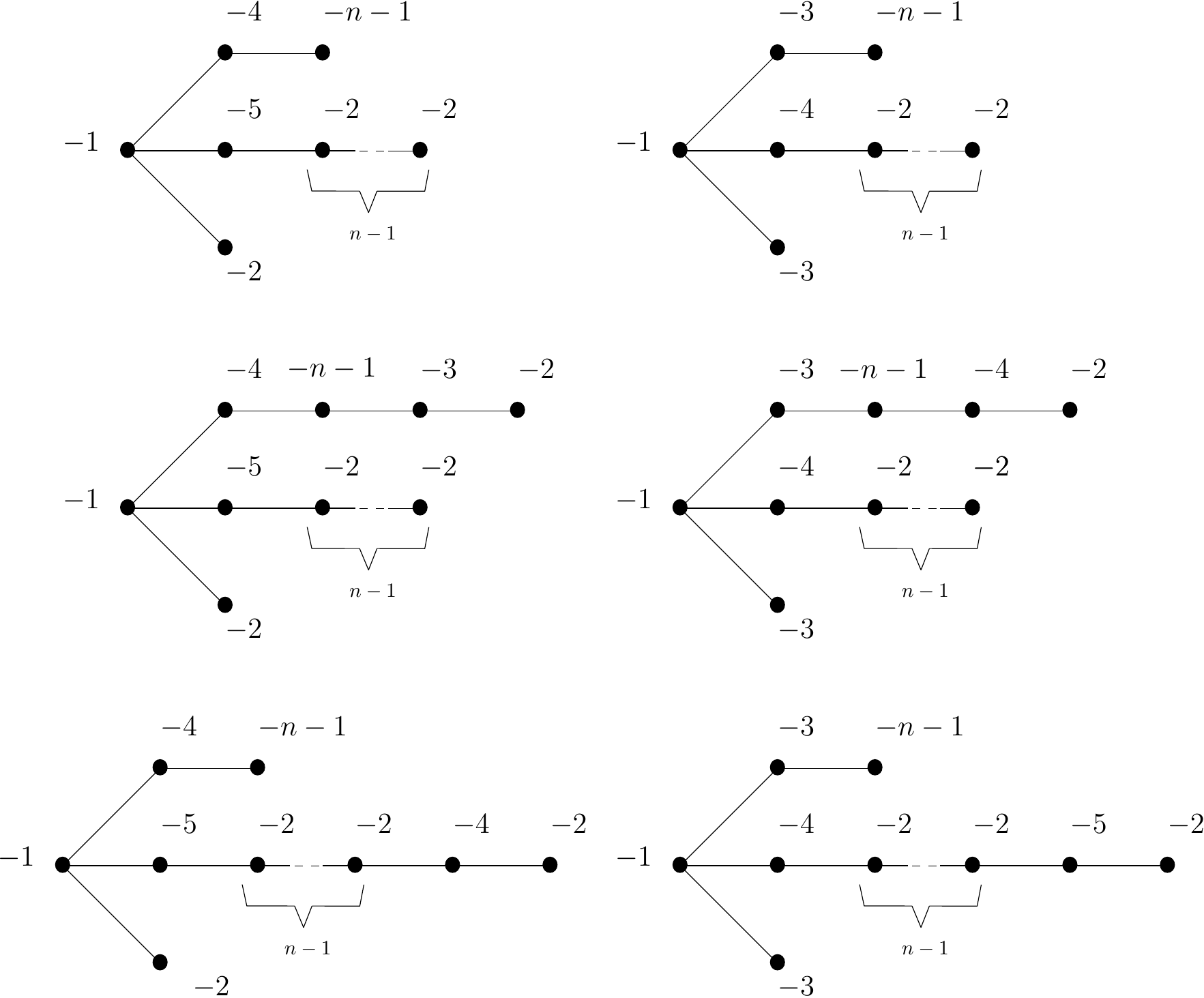}       
		\caption{The plumbing graphs}      \label{fig:plother} 
	\end{center}
\end{figure}

\begin{figure}[ht]  \begin{center}
		\includegraphics[width=1\columnwidth]{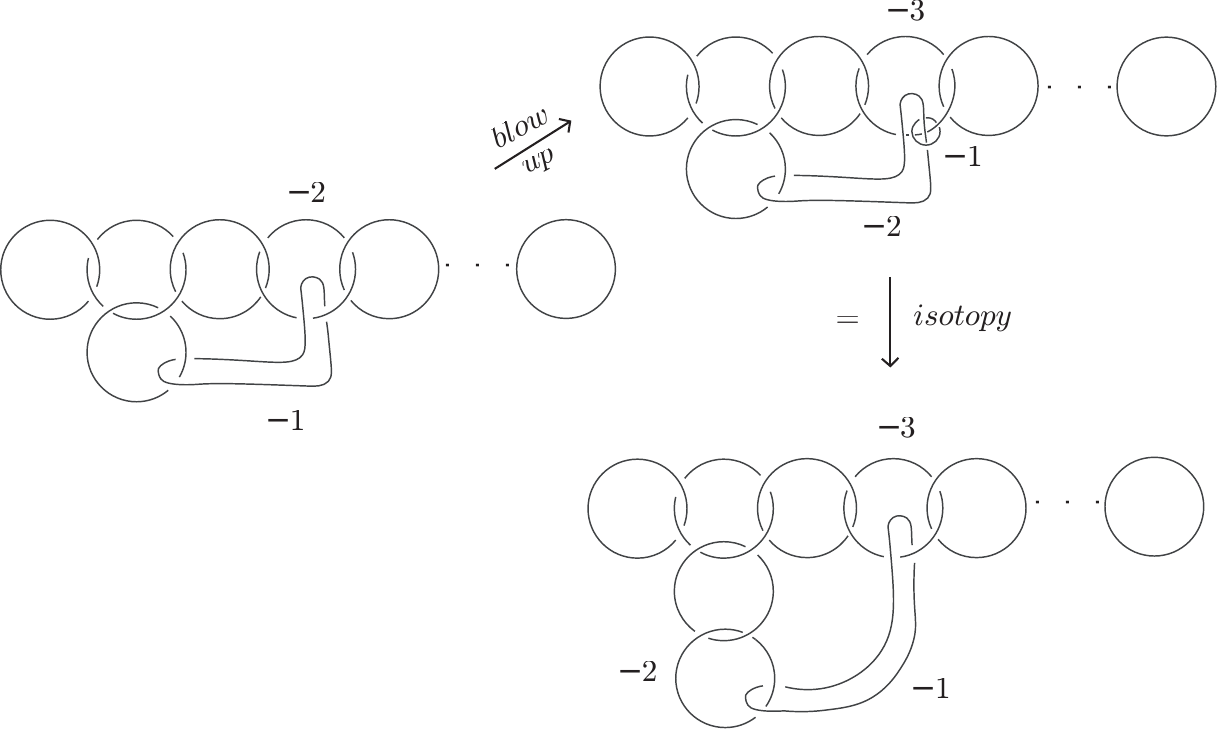}       
		\caption{The Akbulut-Larson trick}      \label{fig:trick} 
	\end{center}
\end{figure}

We initially deal with $\Sigma(2,4n+1,4n+3)$, $\Sigma(3,3n+1,3n+2)$ and $\Sigma(2,4n+1,20n+7)$, $\Sigma(3,3n+1,21n+8)$. The surgery diagrams corresponding to the first element of their plumbing graphs are shown in Figure~\ref{fig:4} and Figure~\ref{fig:5} respectively. Moreover, the dark black $(-1)$-framed components again give the necessary surgery to $Y'$. 

Following the sequence of blow downs in Figure~\ref{fig:4} and Figure~\ref{fig:5}, we reach zero-surgery on the unknot for $\Sigma(2,5,7)$ and $\Sigma(3,4,5)$, and zero-surgery on the stevedore knot for $\Sigma(2,5,27)$ and $\Sigma(3,5,29)$. Then the general families are obtained by applying the \emph{Akbulut-Larson trick\footnote{This is the critical observation of Akbulut and Larson for the proof of their main theorem, which explains the iterative procedure for passing from a surgery diagram to a consecutive one.}} successively, see Figure~\ref{fig:trick}. Each time we add a $(-2)$-framed component to the bottom chain and decrease the framing on the appropriate component in the upper-right chain by $1$. Hence this part of the proof is done.

Since the plumbing graphs of $\Sigma(2,4n+3,20n+13)$ and $\Sigma(3,3n+2,21n+13)$ are quite similar with ones in Theorem \ref{main}, we use the reduction trick. We describe their schematic proofs in Figure~\ref{fig:5} and Figure~\ref{fig:7} respectively. 

In the base cases $n=1$, the essential ingredient of proofs comes from the reduction trick and the remaining blow downs are clearly shown in Figure~\ref{fig:5}. Again we reduce proofs to the base cases $n=1$ from the general cases $n>1$ and we complete them by using the base cases, see Figure~\ref{fig:7}. As blow down operation does not change the boundary three-manifold, this finishes the proof.
\end{proof}

\begin{figure}[ht]  \begin{center}
		\includegraphics[width=.978\columnwidth]{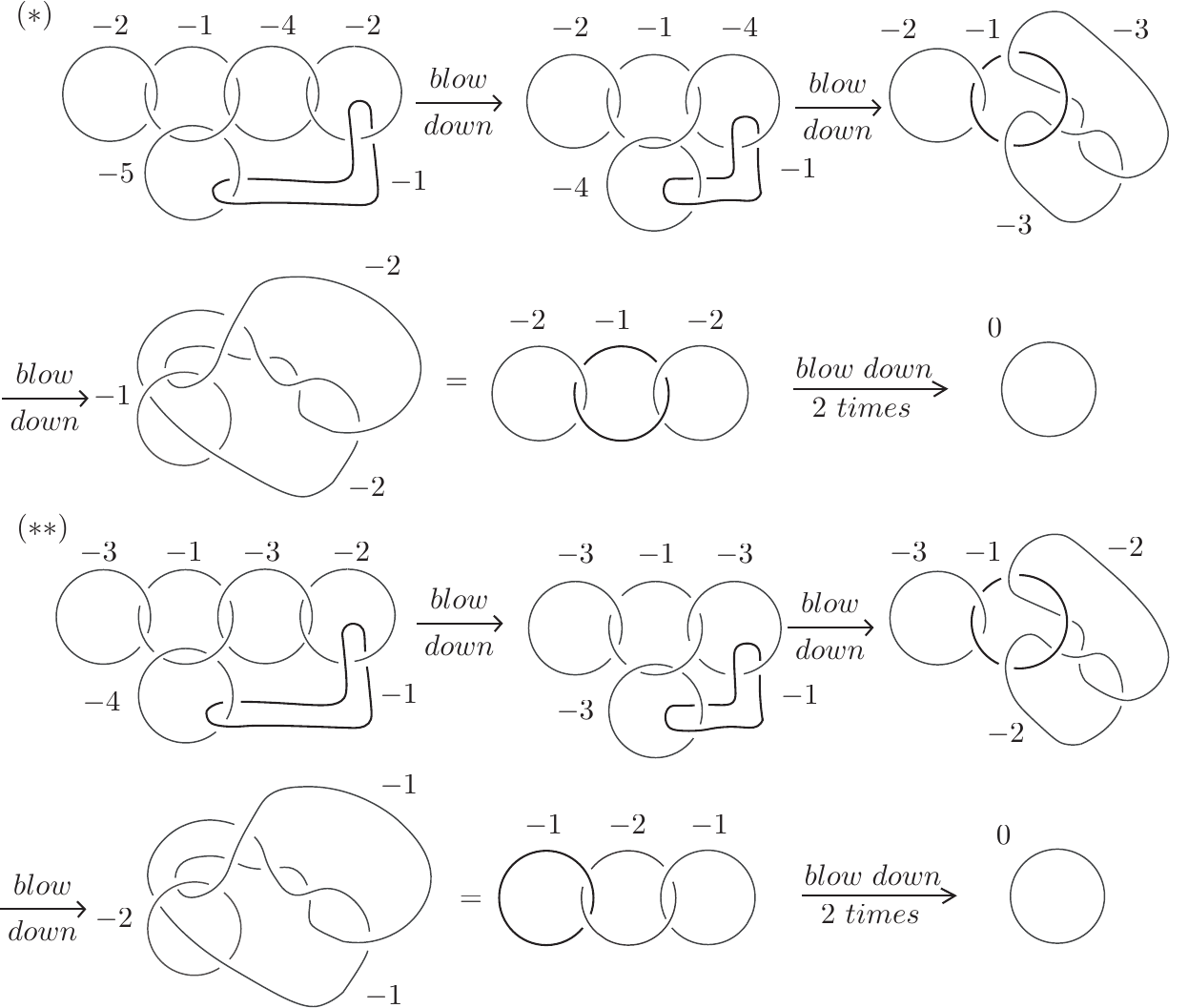}       
		\caption{The $(-1)$-surgeries from $\Sigma(2,5,7)$ and $\Sigma(3,4,5)$ to $Y'$}      \label{fig:4} 
	\end{center}
\end{figure}

\begin{figure}[ht]  \begin{center}
		\includegraphics[width=1\columnwidth]{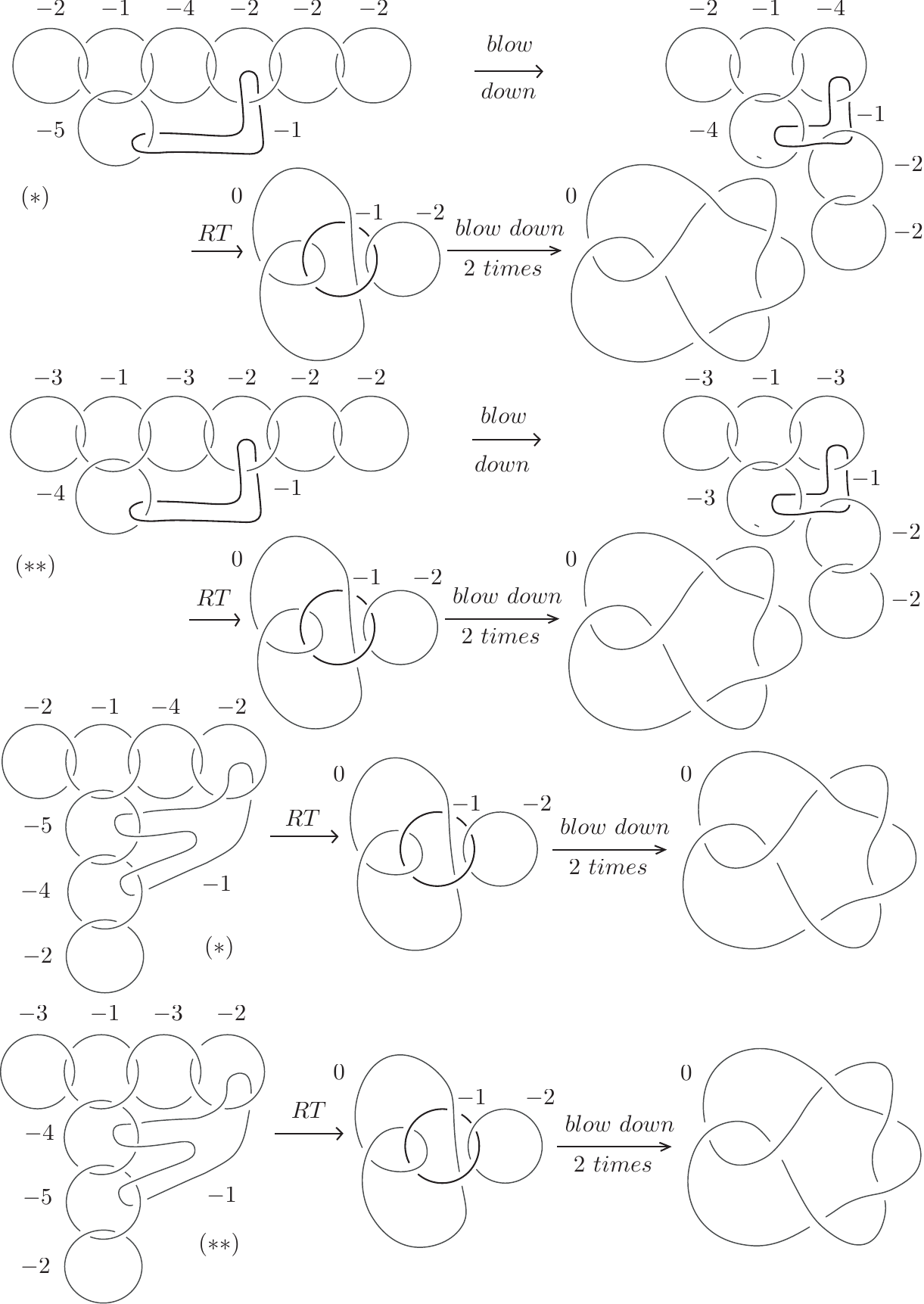}       
		\caption{The $(-1)$-surgeries from $\Sigma(2,5,27)$, $\Sigma(3,5,29)$, $\Sigma(2,7,44)$ and $\Sigma(3,5,34)$ to $Y'$ respectively}      \label{fig:5} 
	\end{center}
\end{figure}

\begin{figure}[ht]  \begin{center}
		\includegraphics[width=1\columnwidth]{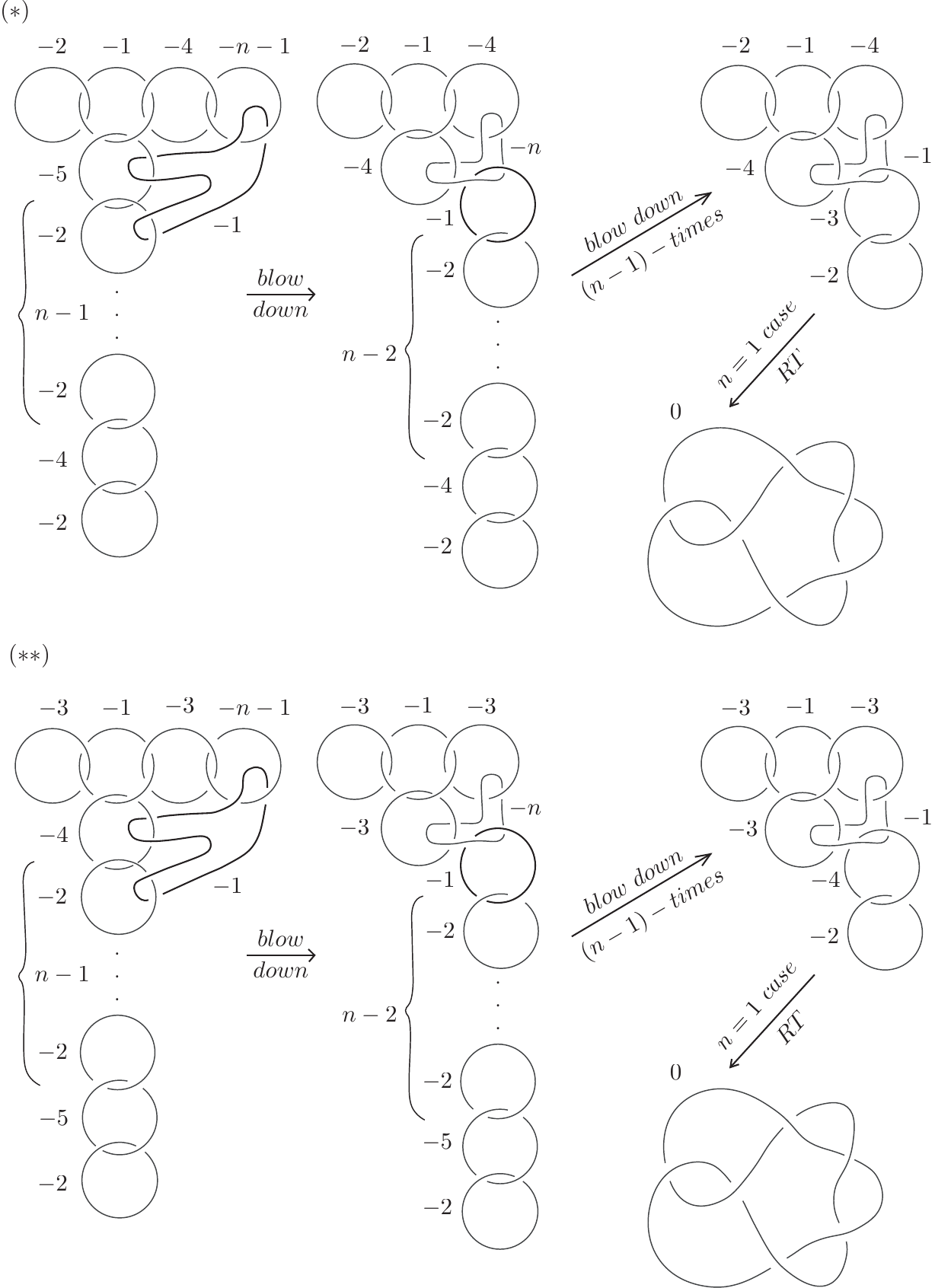}       
		\caption{The sequence of blow downs for $\Sigma(2,4n+3,20n+13)$ and $\Sigma(3,3n+2,21n+13)$}      \label{fig:7} 
	\end{center}
\end{figure}

\clearpage
\bibliography{moreratballs}

\providecommand{\bysame}{\leavevmode\hbox to3em{\hrulefill}\thinspace}
\providecommand{\MR}{\relax\ifhmode\unskip\space\fi MR }
\providecommand{\MRhref}[2]{%
  \href{http://www.ams.org/mathscinet-getitem?mr=#1}{#2}
}
\providecommand{\href}[2]{#2}
\begin{thebibliography}{ACP18}

\bibitem[ACP18]{ACP18}
Paolo Aceto, Daniele Celoria, and JungHwan Park, \emph{Rational cobordisms and
  integral homology}, Preprint (2018),
  \href{https://arxiv.org/abs/1811.01433}{arXiv:1811.01433 [math.GT]}.

\bibitem[AK79]{AK79}
Selman Akbulut and Robion Kirby, \emph{Mazur manifolds}, Michigan Math. J.
  \textbf{26} (1979), no.~3, 259--284.
  \MR{{\href{https://mathscinet.ams.org/mathscinet-getitem?mr=0544597}{0544597}}}

\bibitem[AL18a]{ALa18}
Paolo Aceto and Kyle Larson, \emph{Knot concordance and homology sphere
  groups}, Int. Math. Res. Not. IMRN (2018), no.~23, 7318--7334.
  \MR{{\href{https://mathscinet.ams.org/mathscinet-getitem?mr=3883134}{3883134}}}

\bibitem[AL18b]{AL18}
Selman Akbulut and Kyle Larson, \emph{Brieskorn spheres bounding rational
  balls}, Proc. Amer. Math. Soc. \textbf{146} (2018), no.~4, 1817--1824.
  \MR{{\href{https://mathscinet.ams.org/mathscinet-getitem?mr=3754363}{3754363}}}

\bibitem[CH81]{CH81}
Andrew~J. Casson and John~L. Harer, \emph{Some homology lens spaces which bound
  rational homology balls}, Pacific J. Math. \textbf{96} (1981), no.~1, 23--36.
  \MR{{\href{https://mathscinet.ams.org/mathscinet-getitem?mr=0634760}{0634760}}}

\bibitem[Cha07]{C07}
Jae~Choon Cha, \emph{The structure of the rational concordance group of knots},
  Mem. Amer. Math. Soc. \textbf{189} (2007), no.~885, x+95.
  \MR{{\href{https://mathscinet.ams.org/mathscinet-getitem?mr=2343079}{2343079}}}

\bibitem[Fic84]{F84}
Henry~C. Fickle, \emph{Knots, $\mathbb{Z}$-homology $3$-spheres and
  contractible $4$-manifolds}, Houston J. Math. \textbf{10} (1984), no.~4,
  467–493.
  \MR{{\href{https://mathscinet.ams.org/mathscinet-getitem?mr=774711}{774711}}}

\bibitem[FS81]{FS81}
Ronald Fintushel and Ronald~J. Stern, \emph{An exotic free involution on
  {$S^{4}$}}, Ann. of Math. (2) \textbf{113} (1981), no.~2, 357--365.
  \MR{{\href{https://mathscinet.ams.org/mathscinet-getitem?mr=607896}{607896}}}

\bibitem[FS84]{FS84}
\bysame, \emph{A $\mu$-invariant one homology $3$-sphere that bounds an
  orientable rational ball}, Four-manifold theory (Durham, NH, 1982)
  \textbf{35} (1984), 265--268.
  \MR{{\href{https://mathscinet.ams.org/mathscinet-getitem?mr=0780582}{0780582}}}

\bibitem[GL18]{GL18}
Marco Golla and Kyle Larson, \emph{Linear independence in the rational homology
  cobordism group}, Preprint (2018),
  \href{https://arxiv.org/abs/1803.07931}{arXiv:1803.07931 [math.GT]}.

\bibitem[Kaw79]{Kaw79}
Akio Kawauchi, \emph{The invertibility problem on amphicheiral excellent
  knots}, Proc. Japan Acad. Ser. A Math. Sci. \textbf{55} (1979), no.~10,
  399--402.
  \MR{{\href{https://mathscinet.ams.org/mathscinet-getitem?mr=0559040}{0559040}}}

\bibitem[Kaw09]{Kaw09}
\bysame, \emph{Rational-slice knots via strongly negative-amphicheiral knots},
  Commun. Math. Res. \textbf{25} (2009), no.~2, 177--192.
  \MR{{\href{https://mathscinet.ams.org/mathscinet-getitem?mr=2554510}{2554510}}}

\bibitem[Kir78]{K78}
Robion~C. Kirby, \emph{A calculus for framed links in ${S}^3$}, Invent. Math.
  \textbf{45} (1978), 35--56.
  \MR{{\href{https://mathscinet.ams.org/mathscinet-getitem?mr=0467753}{0467753}}}

\bibitem[KL14]{KL14}
Se-Goo Kim and Charles Livingston, \emph{Nonsplittability of the rational
  homology cobordism group of 3-manifolds}, Pacific J. Math. \textbf{271}
  (2014), no.~1, 183--211.
  \MR{{\href{https://mathscinet.ams.org/mathscinet-getitem?mr=3259765}{3259765}}}

\bibitem[KW18]{KW18}
Min~Hoon Kim and Zhongtao Wu, \emph{On rational sliceness of {M}iyazaki's
  fibered, $-$amphicheiral knots}, Bull. Lond. Math. Soc. \textbf{50} (2018),
  no.~3, 462--476.
  \MR{{\href{https://mathscinet.ams.org/mathscinet-getitem?mr=3829733}{3829733}}}

\bibitem[LT18]{LT18}
Tye Lidman and Eamonn Tweedy, \emph{A note on concordance properties of fibers
  in {S}eifert homology spheres}, Canad. Math. Bull. \textbf{61} (2018), no.~4,
  754--767.
  \MR{{\href{https://mathscinet.ams.org/mathscinet-getitem?mr=3846745}{3846745}}}

\bibitem[Man18]{M18}
Ciprian Manolescu, \emph{Homology cobordism and triangulations}, Proc. Int.
  Cong. of Math. \textbf{1} (2018), 1173--1190,
  {\href{https://eta.impa.br/dl/063.pdf}{ICM 2018, Rio de Janerio}}.

\bibitem[Neu80]{N80}
Walter~D. Neumann, \emph{An invariant of plumbed homology spheres}, {P}roc.
  {S}ympos., {U}niv. {S}iegen, {S}iegen, 1979, Lecture Notes in Math., vol.
  788, Springer, Berlin, 1980, pp.~125--144.
  \MR{{\href{https://mathscinet.ams.org/mathscinet-getitem?mr=0585657}{0585657}}}

\bibitem[NR78]{NR78}
Walter~D. Neumann and Frank Raymond, \emph{Seifert manifolds, plumbing, {$\mu
  $}-invariant and orientation reversing maps}, Algebraic and geometric
  topology ({P}roc. {S}ympos., {U}niv. {C}alifornia, {S}anta {B}arbara,
  {C}alif., 1977), Lecture Notes in Math., vol. 664, Springer, Berlin, 1978,
  pp.~163--196.
  \MR{{\href{https://mathscinet.ams.org/mathscinet-getitem?mr=518415}{518415}}}

\bibitem[Sav02]{Sav02}
Nikolai Saveliev, \emph{Invariants for homology {$3$}-spheres}, Encyclopedia of
  Mathematical Sciences, Low-dimensional topology, vol. 140, Springer-Verlag,
  Berlin, Germany, 2002.
  \MR{\href{https://mathscinet.ams.org/mathscinet-getitem?mr=1941324}{1941324}}

\bibitem[Sie80]{S80}
Laurance Siebenmann, \emph{On vanishing of the {R}okhlin invariant and
  nonfinitely amphicheiral homology {$3$}-spheres}, {P}roc. {S}ympos., {U}niv.
  {S}iegen, {S}iegen, 1979, Lecture Notes in Math., vol. 788, Springer, Berlin,
  1980, pp.~172--222.
  \MR{{\href{https://mathscinet.ams.org/mathscinet-getitem?mr=0585660}{0585660}}}

\bibitem[Ste78]{S78}
Ronald~J. Stern, \emph{Some more {B}rieskorn spheres which bound contractible
  manifolds}, Notices Amer. Math. Soc \textbf{25} (1978), Anouncement,
  {\href{https://www.ams.org/journals/notices/197806/197806FullIssue.pdf}{A448}}.

\end{thebibliography}
\bibliographystyle{amsalpha}

\end{document}